\documentclass[12pt]{amsart}

\usepackage{graphicx}

\usepackage{enumerate}

\newtheorem{thm}{Theorem}[section]
\newtheorem*{thm*}{Theorem}

\newtheorem{conj*}{Conjecture}
\newtheorem{lemma}[thm]{Lemma}
\newtheorem*{lemma*}{Lemma}
\newtheorem{prop}[thm]{Proposition}
\newtheorem*{prop*}{Proposition}
\newtheorem{cor}[thm]{Corollary}

\newtheorem{fact}[thm]{Fact}

\theoremstyle{definition}
\newtheorem{df}[thm]{Definition}

\theoremstyle{remark}
\newtheorem*{rmk}{Remark}

\newtheorem*{claim*}{Claim}

\newcommand{\bbQ}{\mathbb{Q}}

\newcommand{\bbF}{\mathbb{F}}

\newcommand{\bbZ}{\mathbb{Z}}
\newcommand{\bbP}{\mathbb{P}}

\newcommand{\bbR}{\mathbb{R}}

\newcommand{\QQ}{\bbQ}

\newcommand{\RR}{\bbR}

\newcommand{\FF}{\bbF}

\newcommand{\ZZ}{\bbZ}

\newcommand{\PP}{\bbP}

\newcommand{\cO}{\mathcal{O}}

\newcommand{\cH}{\mathcal{H}}

\newcommand{\fp}{\mathfrak p}

\DeclareMathOperator{\Gal}{Gal}

\DeclareMathOperator{\car}{char}

\DeclareMathOperator{\tr}{tr}

\DeclareMathOperator{\ord}{ord}

\DeclareMathOperator{\img}{img}

\DeclareMathOperator{\SL}{SL}
\DeclareMathOperator{\PSL}{PSL}

\DeclareMathOperator*{\mysum}{\sum}

\newcommand{\emphh}[2][ ]{%
\ifthenelse{\equal{#1}{ }}{\index{default}{#2} {\emph{#2}}}{\index{default}{#1@#2} {\emph{#2}}}%
}

\def\sumprime{\mathop{\sum{\raise3pt\hbox{${}'$}}}} 

\newcommand{\tto}[1]{%
\ifthenelse{\equal{#1}{}}{\to}{\stackrel{#1}{\to}}}

\newcommand{\fixme}[1]{}

\usepackage{curves}

\usepackage{amssymb}

\usepackage[all]{xy}
\xyoption{arc}

\usepackage{mathrsfs}
\usepackage{wrapfig, longtable}

\begin{document}

\title[Covers of triangular modular curves and their Galois groups]{Congruence covers of triangular modular curves and their Galois groups}

\author{Luiz Kazuo Takei}

\address{Department of Mathematics, John Abbott College, Sainte-Anne-de-Bellevue, QC, Canada}

\email{luiz.kazuotakei@johnabbott.qc.ca}

\maketitle

\vspace{\baselineskip}

\begin{abstract}
This article computes the Galois groups of congruence covers arising in the context of certain hyperbolic triangle groups. As a consequence of this computation, the genera of the respective curves are deduced.
\end{abstract}

\section{Introduction}

The classical modular group $\SL_2(\ZZ)$ has been extensively studied. In particular, the Galois group of the cover
\[
	\begin{array}{ccc}
		X(p) & \longrightarrow & X(1),
	\end{array}
\]
and the genera of the modular curves $X(p)$ and $X_0(p)$ are all well-known. This article answers the analogous question when the group $\SL_2(\ZZ)$ is replaced by the triangle groups $\Gamma_{q, \infty, \infty}$.

The main results are the computation of the Galois group of the cover
\[
	\begin{array}{ccccc}
		\varphi & : & X_{q, \infty, \infty}(\fp) & \longrightarrow & X_{q, \infty, \infty}(1)
	\end{array}
\]
found in Theorem \ref{thm:quotient_gp}, which is used to deduce the genus of $X_{q, \infty, \infty}(p)$ in Proposition \ref{prop:genus} and finally the computation of the genus of $X_{q, \infty, \infty}^{(0)}(p)$ in Proposition \ref{prop:genus_of_X^0}.

\subsection*{Acknowledgements:} I am grateful to Henri Darmon, Eyal Goren and John Voight for helpful discussions and advice.

\section{Basic definitions and notations}

In this section, a particular family of triangle groups and some important subgroups are defined. For more details and a more general description, cf. section 2 of \cite{takei2012}.

For $q \in \ZZ_{\geq 2}$, the \emph{triangle group} $\Gamma_{q, \infty, \infty}$ is defined to be the subgroup of $\SL_2(\RR)$ generated by
\[
    \gamma_1 = \begin{pmatrix} - 2 \cos(\pi / q) & -1 \\ 1 & 0 \end{pmatrix} \ , \ \gamma_2 = \begin{pmatrix} 0 & 1 \\ -1 & 2 \end{pmatrix},
\]
\[
    \gamma_3 = \begin{pmatrix} 1 & 2 + 2 \cos(\pi/q) \\ 0 & 1 \end{pmatrix}.
\]

In what follows, whenever $\Gamma \subseteq \SL_2(\RR)$, the group $\overline{\Gamma} \subseteq \PSL_2(\RR)$ will denote the image of $\Gamma$ in $\PSL_2(\RR)$.

It is a fact that $\Gamma_{q, \infty, \infty}$ is a Fuchsian group and, moreover, $\left. \overline{\Gamma_{q, \infty, \infty}} \right\backslash \cH^{*} \cong \PP^1$ where $\cH^* = \cH \cup \{ \textrm{cusps of } \Gamma_{q, \infty, \infty} \}$.

Note that $\Gamma_{q, \infty, \infty}$ is a subgroup of $\SL_2(\cO)$, where $\cO = \SL_2(\ZZ[\zeta_{2q} + \zeta_{2q}^{-1}])$ and $\zeta_n = \exp(2 \pi i / n)$.

\begin{df}
    Given a prime ideal $\fp$ of $\cO$, the \textit{congruence subgroups} of $\Gamma_{q, \infty, \infty}$ with level $\fp$ are defined to be
    \[
        \Gamma_{q, \infty, \infty}(\fp) = \left\{ M \in \Gamma_{q, \infty, \infty} \ \left| \ M \equiv \begin{pmatrix} 1 & 0 \\ 0 & 1 \end{pmatrix} \pmod{\fp} \right. \right\}, \textrm{ and }
    \]
    \[
        \Gamma_{q, \infty, \infty}^{(0)}(\fp) = \left\{ M \in \Gamma_{q, \infty, \infty} \ \left| \ M \equiv \begin{pmatrix} * & * \\ 0 & * \end{pmatrix} \pmod{\fp} \right. \right\}.
    \]
\end{df}

\begin{rmk}
	The classical modular group $\SL_2(\ZZ)$ is also a triangle group (in fact, it is the triangle group $\Gamma_{2, 3, \infty}$) and its congruence subgroups as defined above are simply the well known congruence subgroups of $\SL_2(\ZZ)$. For more details, cf. section 2 of \cite{takei2012}.
\end{rmk}

\begin{rmk}
	Note that, unlike $\SL_2(\ZZ)$, almost all of the groups defined above are not arithmetic (cf. \cite{takeuchi1977}).
\end{rmk}

In analogy with the classical case, the \emph{triangular modular curves} associated to those groups are defined as follows:
\begin{equation}
	\begin{array}{cc}
	    X_{q, r, \infty} := \Gamma_{q, r, \infty} \backslash \cH^*, \\
	    X_{q, r, \infty}(\fp) := \Gamma_{q, r, \infty}(\fp) \backslash \cH^*, \\
	    X_{q, r, \infty}^{(0)}(\fp) := \Gamma_{q, r, \infty}^{(0)}(\fp) \backslash \cH^*.
	\end{array}
\label{eqn:def_of_triangular_modular_curves}
\end{equation}

As mentioned in the introduction, one of the goals of this article is the computation of the Galois group of the cover
\[
	\begin{array}{ccccc}
		\varphi & : & X_{q, \infty, \infty}(\fp) & \longrightarrow & X_{q, \infty, \infty}.
	\end{array}
\]

Throughout this article, the following notation will be used:
\[
	\lambda_q = \zeta_{2q} + \zeta_{2q}^{-1} \textrm{ and } \ \mu_q = 2 + \lambda_q
\]
so that
\[
	\gamma_1 = \begin{pmatrix} - \lambda_q & -1 \\ 1 & 0 \end{pmatrix} \ , \ \gamma_2 = \begin{pmatrix} 0 & 1 \\ -1 & 2 \end{pmatrix} \ , \textrm{ and } \gamma_3 = \begin{pmatrix} 1 & \mu_q \\ 0 & 1 \end{pmatrix}.
\]

Moreover,
\begin{equation}
    \label{eqn:rho_map}
    \rho: \Gamma_{q, \infty, \infty} \longrightarrow \SL_2 \left( \frac{\ZZ[\lambda_q]}{\fp} \right)
\end{equation}
denotes the map which sends each matrix to the matrix with reduced entries. It is easy to see that the kernel of this map is exactly $\Gamma_{q, \infty, \infty} (\fp)$. In particular, the group $\Gamma_{q, \infty, \infty} (\fp)$ is normal in $\Gamma_{q, \infty, \infty}$.

\section{Galois group of $\varphi$}

\subsection{Special linear groups over finite fields}

In this section we will use the facts below. Their proofs can either be found in the given reference or be easily deduced.

\begin{fact}
    \label{fct:sl2_order}
    For any prime $p$, we have $|\SL_2(\FF_{p^m})| = (p^m+1)p^m(p^m-1)$.
\end{fact}

\begin{fact}
    \label{fct:sl_presn}
    A presentation for $\SL_2(\FF_{5})$ is given by
    \[
    	\SL_2(\FF_{5}) = \langle x, y \thickspace | \thickspace x^5 = y^3 = (x y)^4 = 1 \rangle.
    \]
    (cf. Example 4, Section 6, Chapter 2 in \cite{suzuki1982})
\end{fact}

\begin{fact}
    \label{fct:general_presn}
    Let $\langle X \thickspace | \thickspace R \rangle$ be a presentation of a group $G$. If $R'$ is another set of relations, then the group $H = \langle X \thickspace | \thickspace R \cup R'  \rangle$ is a homomorphic image of $G$. (cf. Result 6.7, Chapter 2 in \cite{suzuki1982})
\end{fact}

We state a theorem due to Dickson (Theorem 6.17, Chapter 3 in \cite{suzuki1982}).

\begin{thm}
\label{thm:dickson}
    Let $F$ be an algebraically closed field of characteristic $p \geq 2$ and $G$ be a finite subgroup of $\SL_2(F)$ such that $|G|$ is divisible by $p$ and $G$ admits at least two Sylow $p$-subgroups of order $p^r$. Then $G$ is isomorphic to one of the following groups:
    \begin{enumerate}[(i)]
        \item $p = 2$ and $G$ is dihedral of order $2n$ where $n$ is odd
        \item $p = 3$ and $G \cong \SL_2(\FF_{5})$
        \item $\SL_2(K)$
        \item $\left\langle \SL_2(K), d_{\pi} =
                                    \left(
                                    \begin{array}{c c}
                                        \pi & 0 \\
                                        0 & \pi^{-1}
                                    \end{array}
                                    \right) \right\rangle$
    \end{enumerate}
    where $K$ is a field of $p^r$ elements and $\pi$ is an element such that $K(\pi)$ is a field of $p^{2r}$ elements and $\pi^2$ is a generator of $K^{\times}$.
\end{thm}

This allows us to prove the following:

\begin{cor}
    \label{cor:group_gend}
    Let $p \geq 3$ be a prime number and $E = \FF_p(z)$ be the field with $p^m$ elements, where $z \neq 0$. Let $G$ be the subgroup of $\SL_2(E)$ generated by
    \[
        \left(
        \begin{array}{c c}
            0 & 1 \\
            -1 & 2
        \end{array}
        \right)
        \textrm{ \ \ \ and \ \ \ }
        \left(
        \begin{array}{c c}
            1 & z \\
            0 & 1
        \end{array}
        \right)
    \]
    Then,
    \begin{enumerate}[(i)]
        \item $G \cong \SL_2(\FF_{5})$ if $p = 3$ and $z^2 = 2$
        \item $G \cong \SL_2(E)$ otherwise.
    \end{enumerate}
\end{cor}
\begin{proof}
    Denote by $v$ and $u_z$ the matrices
    \[
    \left(
    \begin{array}{c c}
        0 & 1 \\
        -1 & 2
    \end{array}
    \right)
    \textrm{ \ \ \ and \ \ \ }
    \left(
    \begin{array}{c c}
        1 & z \\
        0 & 1
    \end{array}
    \right)
    \]
    respectively. Note that $\ord(v) = \ord(u_z) = p$. We shall prove that $v$ and $u_z$ belong to two distinct Sylow $p$-subgroups.

    \vspace\baselineskip

    \begin{claim*}
        Let $U = \{ u_a : a \in \FF_{q} \} \subseteq \SL_2(E)$ where $q = p^m$ and $u_a = \left( \begin{array}{c c} 1 & a \\ 0 & 1 \end{array} \right)$. Then $U \cap G$ is a $p$-Sylow of $G$. More generally, $U \cap G$ is the only $p$-Sylow of $G$ that contains $u_z$.
    \end{claim*}

    Let $P$ be a $p$-Sylow of $\SL_2(E)$ containing $u_z$. We will prove $P = U$. The claim would then follow by one of the Sylow theorems (namely the one which says that any $p$-subgroup is contained in a $p$-Sylow). In fact, by one of the Sylow theorems, $P = \alpha U \alpha^{-1}$ for some $\alpha = \left( \begin{array}{c c} a & b \\ c & d \end{array} \right) \in \SL_2(\FF_{q})$ (because $U$ is a $p$-Sylow of $\SL_2(\FF_{q})$). So, there exists $z' \in \FF_q \backslash \{ 0 \}$ such that $u_z = \alpha u_{z'} \alpha^{-1} = \left( \begin{array}{cc} 1- acz' & a^2 z \\ - c^2 z & 1 + acz \end{array} \right)$. Hence, $c = 0$. Thus, $P = U$.

    \vspace\baselineskip

    Hence, $v$ and $u_z$ belong to two distinct $p$-Sylows of $G$. Therefore, we can use \ref{thm:dickson}.

    Since we are assuming $p > 2$, there are only 3 possibilities for $G$: $\SL_2(\FF_{5})$ (this can only happen when $p=3$), $\SL_2(\FF_{p^r})$ or $\langle \SL_2(\FF_{p^r}) , d_{\pi} \rangle$, where $p^r$ is the order a $p$-Sylow of $G$.

    \vspace\baselineskip

    \begin{claim*}
        $G \not\cong \langle \SL_2(\FF_{p^r}) , d_{\pi} \rangle$.
    \end{claim*}

    Let $H = \langle \SL_2(\FF_{p^r}) , d_{\pi} \rangle$. If $G \cong H$, then their respective abelianizations are also isomorphic: $G^{\textrm{ab}} \cong H^{\textrm{ab}}$. Since $G = \langle v, u_z\rangle$ and $\ord(v) = \ord(u_z) = p$, every element of $G^{\textrm{ab}}$ has order $p$. We claim that $\overline{d_{\pi}}$ (the image of $d_{\pi}$ in $H^{\textrm{ab}}$) can't have order $p$.

    In fact, $\ord(d_{\pi}) = \ord(\pi)$. We know that $\ord(\pi^2) = p^r - 1$. On the other hand,
    \[
    \ord(\pi) = \left\{
                    \begin{array}{c l l}
                        \ord(\pi^2) & , & \textrm{ if } \ord(\pi) \textrm{ is odd} \\
                        2 \ord(\pi^2) & , & \textrm{ if } \ord(\pi) \textrm{ is even}
                    \end{array}
                \right.
    \]
    So, $\ord(\pi) = (p^r - 1)$ or $2 (p^r - 1)$.

    Since $\ord(\overline{d_{\pi}}) \mid \ord(d_{\pi})$ and $p \nmid 2 (p^r - 1)$, $\ord(\overline{d_{\pi}}) \neq p$.

    \vspace\baselineskip

    It remains to show that if $p = 3$, then $G \cong \SL_2(\FF_{5})$ if and only if $z^2 = 2$.

    \vspace\baselineskip

    \begin{claim*}
        If $p = 3$ and $G \cong \SL_2(\FF_{5})$, then $z^2 = 2$.
    \end{claim*}

    By the corollaries of Theorem 9.8, Chapter 1 in \cite{suzuki1982}, we have $Z(G) = \{ \pm I \}$ and, thus, $|Z(G)| = 2$. So, by the corollary of Theorem 9.9, Chapter 1 in \cite{suzuki1982}, $\frac{G}{Z(G)}$ is a simple group of order $60$. Therefore, by Exercise 9, Section 3, Chapter 3 in \cite{suzuki1982}, $\frac{G}{Z(G)} \cong A_5$.

    Let $\overline{v}$ and $\overline{u_z^{-1}}$ be the images of $v$ and $u_z^{-1}$ respectively in $A_5$. Since $v$ and $u_z^{-1}$ are clearly not in $Z(G)$ and their order is $3$, we get that $\ord(\overline{v}) = \ord(\overline{u_z^{-1}}) = 3$. So, $\overline{v} = (a b c)$ and $\overline{u_z^{-1}} = (d e f)$. Obviously we need to have $\{ a, b, c, d , e ,f \} = \{ 1, 2, 3, 4, 5 \}$. Without loss of generality, $\overline{v} = (1 2 3)$ and $\overline{u_z^{-1}} = (1 4 5)$. So, $\overline{v} \overline{u_z^{-1}} = (1 2 3 4 5)$. So, $(v u_z^{-1})^{5} = \pm I$. Thus, $\ord(v u_z^{-1}) = 5$ or $10$. Hence, looking at the Jordan canonical form of $v u_z^{-1}$, we get
    \[
    z+2 = \tr(v u_z^{-1}) = \pm (x + x^{-1})
    \]
    for some $x$ primitive fifth root of unity over $\FF_3$.

    Since $x^4 + x^3 + x^2 + x + 1 = 0$,
    \[
    (z+2)^2 = \mp (z+2) + 1 \textrm{ \ \ \ , i.e., \ \ \ } z^2 = z + 1 \textrm{ or } z^2 = 2
    \]

    If $z^2 = z + 1$, then $G = \langle v, u_z \rangle$ has 720 elements\footnotemark[1]. Hence, $z^2 = 2$.

    \vspace\baselineskip

    \begin{claim*}
        If $z^2 = 2$, then $G \cong \SL_2(\FF_{5})$.
    \end{claim*}

    We also have\footnotemark[1] that $|G| = 120$. So, $|G| = |\SL_2(\FF_{5})|$ (by fact \ref{fct:sl2_order}). Let $h = (v u_z)^2 = \left( \begin{array}{cc} 2 & 2+2z \\ z+1 & 2+2z \end{array} \right)$. Notice $h^5 = u^3 = (h u)^4 = 1$. Moreover, $\alpha = h^{-1} u h$. Hence, $G = \left< h, u \right>$. Then, since $G = |\SL_2(\FF_{5})|$, by facts \ref{fct:sl_presn} and \ref{fct:general_presn} we obtain $G \cong \SL_2(\FF_{5})$.
\end{proof}

\footnotetext[1]{Verified using the computer algebra system Sage
\cite{sage}.}

\subsection{A bit of algebraic number theory}

Our goal in this section is to state some basic facts from algebraic number theory and give an explicit formula for $f(\mathfrak{B}^+ | p)$ that will be used later. Unless otherwise mentioned, the facts and definitions in this section can be found in most algebraic number theory textbooks like \cite{marcus1977} and \cite{frohlich&taylor1993}.

Let us start fixing our notation:
\begin{itemize}
    \renewcommand{\labelitemi}{$\cdot$}
    \itemsep0em
    \item $e(- \mid -)$ := ramification index of one prime above another one,
    \item $f(- \mid -)$ := inertia degree of one prime above another one,
    \item $r(- \mid -)$ := number of distinct primes above a given one at the base field.
\end{itemize}

Moreover, in this section, we assume that
\begin{itemize}
    \renewcommand{\labelitemi}{$\cdot$}
    \itemsep0em
    \item $q$ is an odd number and $p$ is a prime number,
    \item $L_{q} := \QQ(\zeta_q)$ is the $q$-th cyclotomic field and $L_q^+ := \QQ(\zeta_q + \zeta_q^{-1})$,
    \item $\mathfrak{B}^+$ is a prime in $\mathcal{O}_{L_q^+}$ above $p$ and $\mathfrak{B}$ is a prime in $\mathcal{O}_{L_q}$ above $\mathfrak{B}^+$,
    \item $r$ is the number of primes in $\mathcal{O}_{L_q}$ above $p$ and $r^+$ is the number of primes in $\mathcal{O}_{L_q^+}$ above $p$,
    \item $D(\mathfrak{B}) = D(\mathfrak{B} | \QQ) = \{ \sigma \in \Gal(L_q | \QQ) : \sigma \mathfrak{B} = \mathfrak{B} \}$ is the decomposition group of $\mathfrak{B}$ over $\QQ$ and $K^D$ is the respective decomposition field (i.e., the subfield of $L_q$ that is fixed by $D(\mathfrak{B})$).
\end{itemize}

\begin{rmk}
    Notice $L_q | \QQ$ is a Galois extension and $L_q^+$ is the field fixed by $H := \{ 1, -1 \}$ (where $-1$ is the complex conjugation). In particular, $[L_q: L_q^+] = 2$.
\end{rmk}

\vspace\baselineskip

\begin{fact}
    \label{fct:number_of_primes_above}
    $r(L_q | \mathfrak{B}^+)$ = number of elements in the $H$-orbit of $\mathfrak{B}$.
\end{fact}

\begin{fact}
    \label{fct:decomp_in_cyclotomic_fields}
    If $p \nmid q$, then
    \begin{enumerate}[(i)]
        \item $e(\mathfrak{B} | p) = 1$ (hence, $e(\mathfrak{B}^+ | p) = e(\mathfrak{B} | \mathfrak{B}^+) = 1$)
        \item $f(\mathfrak{B} | p) = f$, where $f$ is the order of $\overline{p}$ in $\left( \frac{\ZZ}{q \ZZ} \right) ^*$
        \item $r = \frac{\phi(q)}{f}$, where $\phi$ is Euler's phi function.
    \end{enumerate}
\end{fact}

\begin{fact}
    \label{fct:decomp_in_galois_extensions}
    Let $R$ be a Dedekind domain, $K$ its field of fractions, $L$ a finite Galois extension of $K$, $\mathcal{O}_L$ the ring of integers of $L$ and $G = \Gal(L | K)$. Let $\mathfrak{p}$ be a prime ideal in $R$ and $\mathfrak{B}_1, \dotsc, \mathfrak{B}_s$ the distinct prime ideals in $\mathcal{O}_L$ above $\mathfrak{p}$. Then
    \begin{enumerate}[(i)]
        \item $e(\mathfrak{B}_1 | \mathfrak{p}) = \dotsb = e(\mathfrak{B}_s | \mathfrak{p})$ and $f(\mathfrak{B}_1 | \mathfrak{p}) = \dotsb = f(\mathfrak{B}_s | \mathfrak{p})$
        \item $e(\mathfrak{B}_j | \mathfrak{p}) \cdot f(\mathfrak{B}_j | \mathfrak{p}) \cdot s = [L:K]$ (for every $j = 1, \dotsc, s$).
    \end{enumerate}
\end{fact}

\begin{fact}
    \label{fct:decomp_field}
    $[\Gal(L_q | \QQ) : D(\mathfrak{B})] = r$
\end{fact}

\begin{fact}
    \label{fct:image_of_decomp_group}
    The map $\Gal(L_q | \QQ) \rightarrow \Gal \left( \frac{\mathcal{O}_{L_q}}{\mathfrak{B}} \big| \frac{\ZZ}{p} \right)$ defined by $\sigma \mapsto \overline{\sigma}$ is surjective.
\end{fact}

\begin{fact}
    \label{fct:group_of_units_power}
    If $l$ is a prime and $\alpha$ is a positive integer, then $\left( \frac{\ZZ}{l^{\alpha} \ZZ} \right)^*$ is a cyclic group of order $\phi(l^{\alpha}) = l^{\alpha - 1} (l - 1)$.
\end{fact}

\begin{fact} (\emph{Theorem 2.13 in \cite{washington1982}})
    \label{fct:cyclotomic_inertia_degree}
    If $(p,q)=1$, then $f(\mathfrak{B} | p)$ is the smallest positive integer $f$ such that $p^f \equiv 1 \pmod{q}$.
\end{fact}

\vspace\baselineskip

\begin{prop}
    \label{prop:f+_and_f}
    Suppose $(p,q) = 1$. Then,
    \[
        f(\mathfrak{B}^+ | p) = \left\{
                                    \begin{array}{lcl}
                                        f(\mathfrak{B} | p) & , & \textrm{ if } -1 \not\in D(\mathfrak{B}) \\
                                        \frac{f(\mathfrak{B} | p)}{2} & , & \textrm{ if } -1 \in D(\mathfrak{B})
                                    \end{array}
                                \right.
    \]
\end{prop}
\begin{proof}
    Since the inertia degree is multiplicative, we get that $f(\mathfrak{B}^+ | p) = f(\mathfrak{B} | p)$ if and only if f($\mathfrak{B} | \mathfrak{B}^+) = 1$. By fact \ref{fct:decomp_in_galois_extensions} and \ref{fct:decomp_in_cyclotomic_fields}, we have $f(\mathfrak{B} | \mathfrak{B}^+) \cdot r(L_q | \mathfrak{B}^+) = 2$. So, it is enough to show that $r(L_q | \mathfrak{B}^+) = 1$ if and only if $-1 \in D(\mathfrak{B})$. But this follows easily from fact \ref{fct:number_of_primes_above}.
\end{proof}

\begin{prop}
    \label{prop:group_of_decomp}
    $D(\mathfrak{B}) \cong \frac{\ZZ}{f \ZZ}$
\end{prop}
\begin{proof}
    We want to use fact \ref{fct:image_of_decomp_group} to show this.

    Let us start by noting that $|D(\mathfrak{B})| = f$. In fact, by facts \ref{fct:decomp_field} and \ref{fct:decomp_in_cyclotomic_fields}, we have that $\frac{|Gal(L_q | \QQ)|}{|D(\mathfrak{B})|} = r = \frac{\phi(q)}{f}$. Since $|\Gal(L_q | \QQ)| = \phi(q)$, we obtain what we claimed.

    Notice now that, by definition of inertia degree, $\frac{\mathcal{O}_{L_q}}{\mathfrak{B}} = \FF_{p^f}$.

    Hence, since $\Gal(\FF_{p^f} \mid \FF_p) = \frac{\ZZ}{f \ZZ}$, we obtain what we wanted.
\end{proof}

Let us now prove a particular case of our main goal.

\begin{lemma}
    If $q = l^{\alpha}$ is a prime power, then $-1 \in D(\mathfrak{B})$ if and only if $f$ is even.
\end{lemma}
\begin{proof}
    Suppose $f$ is odd. Since $-1$ is an element of order $2$, Proposition \ref{prop:group_of_decomp} tells us that $-1 \not\in D(\mathfrak{B})$.

    Now suppose $f$ is even. So, Sylow's Theorem and Proposition \ref{prop:group_of_decomp} says that $D(\mathfrak{B})$ has at least one element of order $2$. But since $\Gal(L_q | \QQ) = \left( \frac{\ZZ}{q \ZZ} \right)^{\times}$ is cyclic (fact \ref{fct:group_of_units_power}), it has only one element or oder $2$, namely $-1$.
\end{proof}

\begin{prop}
    If $(p,q) = 1$ then $f(\mathfrak{B}^+ | p)$ is the smallest positive integer $f^+$ such that $p^{f^+} \equiv \pm 1 \pmod{q}$.
\end{prop}
\begin{proof}
    By Proposition \ref{prop:f+_and_f} and the previous lemma, we have that
    \[
        f(\mathfrak{B}^+ | p) = \left\{
                                    \begin{array}{lcl}
                                        f(\mathfrak{B} | p) & , & \textrm{ if } f(\mathfrak{B} | p) \textrm{ is odd} \\
                                        \frac{f(\mathfrak{B} | p)}{2} & , & \textrm{ if }  f(\mathfrak{B} | p) \textrm{ is even}.
                                    \end{array}
                                \right.
    \]

    The result now follows from fact \ref{fct:cyclotomic_inertia_degree}.
\end{proof}

We are finally ready to tackle the general case:
\begin{prop}
    \label{prop:inertia_degree_of_p}
    If $p$ is any prime number, than $f(\mathfrak{B}^+ | p)$ is the smallest positive integer $f^+$ such that $p^{f^+} \equiv \pm 1 \pmod{q'}$, where $2q = p^a q'$ and $(p,q') = 1$.
\end{prop}
\begin{proof}
    This follows from the previous proposition and the fact that $p$ is totally ramified in $L_p$ (Lemma 1.4 in \cite{washington1982}).
\end{proof}

\subsection{Computing $\left[ \overline{\Gamma_{q, \infty, \infty}} : \overline{\Gamma_{q, \infty, \infty}(\fp)} \right]$}

Let $\rho$ be the map defined in (\ref{eqn:rho_map}). We define
\[
\begin{array}{ccccc}
    \overline{\rho} & : & \overline{\Gamma_{q, \infty, \infty}} & \longrightarrow & \PSL_2(E) \\
                    &   & \overline{g} & \longmapsto    & \overline{\rho(g)}
\end{array}
\]
where $E = \frac{\ZZ[\lambda_q]}{\fp}$ and $\bar{ }$ denotes the image in $\PSL_2$. This map is well-defined because $\rho(-g) = - \rho(g)$.

Therefore, we have a commutative diagram

\centerline{
    \xymatrix{
        \Gamma_{q, \infty, \infty} \ar[rr]^{\rho \ \ } \ar@{->>}[d] & &  \SL_2(E) \ar@{->>}[d] \\
        \overline{\Gamma_{q, \infty, \infty}} \ar[rr]_{\overline{\rho} \ \ } & & \PSL_2(E)
    }
}

\begin{lemma}
    $\ker(\overline{\rho}) = \overline{\Gamma_{q, \infty, \infty}(\fp)}$.
\end{lemma}
\begin{proof}
    Since $\ker(\rho) = \Gamma_{q, \infty, \infty}(\fp)$, it is clear that $\overline{\Gamma_{q, \infty, \infty}(\fp)} \subseteq \ker(\overline{\rho})$.

    Now, take $\overline{g} \in \ker(\overline{\rho})$, i.e., $\rho(g) = \pm I$ ($I$ is the identity matrix). Since $\rho(-g) = -\rho(g)$, we get $\pm g \in \ker(\rho) = \Gamma_{q, \infty, \infty}(\fp)$. So, $g \in \pm \Gamma_{q, \infty, \infty}(\fp)$. Hence, $\overline{g} \in \overline{\Gamma_{q, \infty, \infty}(\fp)}$.
\end{proof}

This shows that $\overline{\Gamma_{q, \infty, \infty}} \ / \ \overline{\Gamma_{q, \infty, \infty}(\fp)} \ \cong \ \img(\overline{\rho})$.

\begin{fact}
    \label{sl2_center}
    The center $Z(\SL_2(F))$ of $\SL_2(F)$ (where $F$ is any field) is equal to $\{ \pm I \}$. (cf. Corollary 2 of Result 9.8, Chapter 1 in \cite{suzuki1982})
\end{fact}

\begin{fact}
    \label{finite_quotient}
    If $\mathfrak{a} \subseteq \ZZ[\lambda_q]$ is a non-zero ideal, then $\frac{\ZZ[\lambda_q]}{\mathfrak{a}}$ is finite.
\end{fact}

\begin{fact}
    \label{dihedral_center}
    If $n$ is odd, then $Z(D_{2n}) = \{ e \}$.
\end{fact}

\begin{lemma}
    \label{lem:quotient_gp_p=2}
    If $\fp$ is a prime ideal lying above $2 \ZZ$ and $q$ is odd, then $\img(\overline{\rho}) \cong D_{2s}$ (for some odd $s$ dividing $q$). Moreover, if $q$ is a prime, then $s = q$.
\end{lemma}
\begin{proof}
    Notice $\img(\overline{\rho}) = \overline{\img(\rho)}$.

    If $p = 2$, one can easily check that $\ord(\rho(\gamma_2)) = \ord(\rho(\gamma_3)) = 2$. Hence, by the first claim in the proof of corollary \ref{cor:group_gend}, $\rho(\Gamma_{q, \infty, \infty})$ has 2 distinct $2$-Sylow subgroups. Hence, $\rho(\Gamma_{q, \infty, \infty})$ is one of the groups listed in Theorem \ref{thm:dickson}.

    One can check that $\rho(\gamma_1) \rho(\gamma_3) = \rho(\gamma_3) \rho(\gamma_1)^{-1}$ and $\rho(\gamma_3)^2 = \rho(\gamma_1)^q = I$. Hence, since $\rho(\Gamma_{q, \infty, \infty}) = \langle \rho(\gamma_1) , \rho(\gamma_3) \rangle$, fact \ref{fct:general_presn} tells us that $\rho(\Gamma_{q, \infty, \infty})$ is a homomorphic image of $D_{2q}$. In particular, $|\rho(\Gamma_{q, \infty, \infty})| \mid 2q$. Since $q$ is odd, by Theorem \ref{thm:dickson}, $\rho(\Gamma_{q, \infty, \infty})$ can only be $D_{2n}$ (for some odd $n$) or $\SL_2(2)$. But, one can check that $\SL_2(2) = D_{2 \cdot 3}$. So, in any case, $\rho(\Gamma_{q, \infty, \infty}) \cong D_{2s}$ (for some odd $s$).

    Now, since $\car(E) = 2$, $I = -I$ in $\SL_2(E)$. Hence, $\PSL_2(E) = \SL_2(E)$. Thus, $\overline{\img(\rho)} \cong \img(\rho) \cong D_{2s}$. So, $2s \mid 2q$. Therefore, since $q$ is odd, $s \mid q$.

    Finally, since $\gamma_1 = \left( \begin{array}{cc} -\lambda_q & -1 \\ 1 & 0 \end{array} \right)$, $\rho(\gamma_1) \neq I$ (because $0 \not\equiv 1 \pmod{\fp}$). So, if $q$ is prime, $\ord(\rho(\gamma_1)) = q$ (because $\ord(\rho(\gamma_1)) \mid \ord(\gamma_1) = q$).
\end{proof}

\begin{lemma}
    \label{lem:quotient_gp}
    Suppose $\fp$ is a prime ideal lying above $p \ZZ$ with $p \geq 3$. Then, $\img(\overline{\rho})$ is isomorphic to
    \begin{enumerate}[(i)]
        \item $\PSL_2(\FF_{5})$, if $p = 3$ and $\mu_q^2 - 2 \in \fp$
        \item $\PSL_2(E)$, otherwise (where $E = \frac{\ZZ[\lambda_q]}{\fp}$)
    \end{enumerate}
\end{lemma}
\begin{proof}
    Notice $\img(\overline{\rho}) = \overline{\img(\rho)}$.

    \vspace\baselineskip

    If $p = 3$ and $mu_q^2 - 2 \in \fp$, fact \ref{finite_quotient} and corollary \ref{cor:group_gend} says that $\img(\rho) \cong \SL_2(\FF_{5})$. We have to prove that $\overline{\img(\rho)} \cong \PSL_2(\FF_{5})$. Notice that $\overline{\img(\rho)} = \frac{\img(\rho)}{ \{ \pm I \} \cap \img(\rho) }$.

    We can verify that $- I \in \img{\rho}$. In fact, there are only two cases to consider and they were computed explicitly using Sage \cite{sage}.

    So, $\overline{\img(\rho)} = \frac{\SL_2(\FF_{5})}{Z(\SL_2(\FF_{5}))} = \PSL_2(\FF_{5})$.

    \vspace\baselineskip

    Otherwise, fact \ref{finite_quotient} and corollary \ref{cor:group_gend} says that $\img(\rho) = \SL_2(E)$, i.e., $\rho$ is surjective. Hence, $\overline{\img(\rho)} = \PSL_2(E)$.
\end{proof}

\begin{thm}
\label{thm:quotient_gp}
    Let $q \geq 3$ be an odd integer and $\fp$ be a prime ideal of $\ZZ[\lambda_q]$ lying above $p \ZZ$ where $p \geq 2$. Moreover, let $G$ denote the Galois group of
    \[
		\begin{array}{ccccc}
			\varphi & : & X_{q, \infty, \infty}(\fp) & \longrightarrow & X_{q, \infty, \infty}.
		\end{array}
    \]
    \begin{enumerate}[(i)]
        \item If $p = 2$, then $G \cong D_{2s}$ (for some odd $s$ that divides $q$) and, hence, $\left[ \overline{\Gamma_{q, \infty, \infty}} : \overline{\Gamma_{q, \infty, \infty}(\fp)} \right] = 2s$. Moreover, if $q$ is prime, then $s = q$.
        \item If $p = 3$ and $\mu_q^2 - 2 \in \fp$, then $G \cong \PSL_2(\FF_{5})$ and, hence, $\left[ \overline{\Gamma_{q, \infty, \infty}} : \overline{\Gamma_{q, \infty, \infty}(\fp)} \right] = 60$;
        \item Otherwise, $G \cong \PSL_2(\ZZ[\lambda_q]/\fp)$ and, hence, $\left[ \overline{\Gamma_{q, \infty, \infty}} : \overline{\Gamma_{q, \infty, \infty}(\fp)} \right] = (p^m + 1) p^m (p^m - 1) / 2$.
    \end{enumerate}

    Moreover $\ZZ[\lambda_q] / \fp \cong \mathbb{F}_{p^m}$, where $m$ is the smallest positive integer such that $p^m \equiv \pm 1 \pmod{q'}$ ($2q = p^a q'$ with $\gcd(p,q') = 1$).
\end{thm}
\begin{proof}
    The theorem follows from Lemmas \ref{lem:quotient_gp_p=2} and \ref{lem:quotient_gp}.

    The fact that $\ZZ[\lambda_q] / \fp$ is a field with $p^m$ elements with $m$ as in the statement of the theorem follows from Proposition \ref{prop:inertia_degree_of_p}.
\end{proof}

\begin{rmk}
	Recall that, in the classical setting, the Galois group of
	\[
		X(p) \rightarrow X(1)
	\]
	is always $\PSL_2(\ZZ / p)$. The previous theorem shows that is not always the case for a general triangle group and, in fact, establishes exactly when that happens for the triangle groups $\Gamma_{q, \infty, \infty}$.
\end{rmk}

\begin{rmk}
	Note that item (iii) of the previous theorem was also deduced independently in Theorem 8.1 of \cite{clark&voight2011}.
\end{rmk}

\section{Genus of $X_{q, \infty, \infty}(\fp)$}

The genera of the curves associated to the Hecke triangle groups (that is, the groups $\Gamma_{2,r, \infty}$) were computed in \cite{llt2000}. We deal in a similar way with the triangle groups $\Gamma_{q, \infty, \infty}$.

\vspace{\baselineskip}

The proposition below paired with Theorem \ref{thm:quotient_gp} computes the genera of $X_{q, \infty, \infty}(\fp)$ for many ideals $\fp$.

\begin{prop}
\label{prop:genus}
    Suppose $q$ is an odd prime number and $\fp$ is a prime ideal of $\ZZ[\lambda_q]$ lying above $p \ZZ$. Suppose also that $p \neq q$. Then the genus of $X_{q, \infty, \infty}(\fp)$ is
    \[
    1 + \frac{\overline{\mu}}{2} \left( 1 - \frac{2}{p} - \frac{1}{q} \right),
    \]
    where $\overline{\mu} = \left[ \overline{\Gamma_{q, \infty, \infty}} : \overline{\Gamma_{q, \infty, \infty}(\fp)} \right]$.
\end{prop}
\begin{proof}
    To simplify notation, let us call $\Gamma = \Gamma_{q, \infty, \infty}$.

    We know that the map
    \[
    	\varphi : \Gamma(\fp) \backslash \mathcal{H}^* \longrightarrow \Gamma \backslash \mathcal{H}^*
    \]
    is holomorphic and has degree $\overline{\mu}$ (cf. Section 1.5 in \cite{shimura1994iat}). So we can use Riemann-Hurwitz formula to compute the genus $g$ of $\Gamma(\fp) \backslash \mathcal{H}^*$:
\[
\begin{array}{r c l}
    2g - 2 & = & \overline{\mu} (2 \cdot 0 - 2) + \mysum\limits_{P \in \Gamma(J) \backslash \mathcal{H}^*} (e_P - 1) \\
            & = &  - 2 \overline{\mu} + \mysum\limits_{P \in \Gamma(J) \backslash \mathcal{H}^*} (e_P - 1)
\end{array}
\]
where $e_P$ is the ramification index of $\varphi$ at $P$.

    By Proposition 1.37 in \cite{shimura1994iat}, we see that the only points $P$ which may have $e_P > 1$ are the points which are mapped to cusps or elliptic points.

    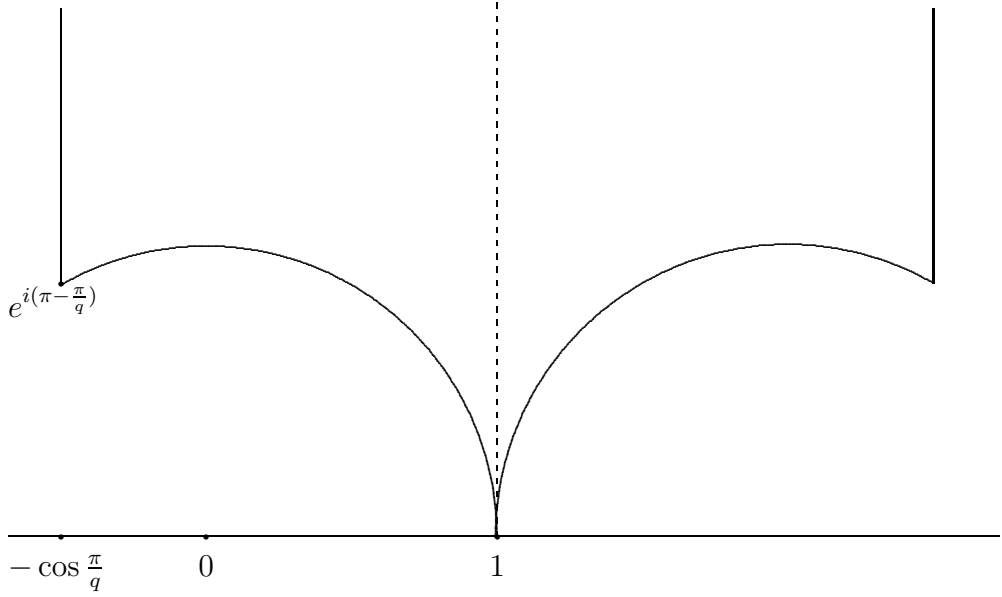
\begin{figure}[!h]
    \begin{center}
    \begin{picture}(375,215)(-75,-15)   
        \put(-75,0){\line(1,0){375}}

        \put(-3,-15){$0$}
        \put(0,0){\circle*{2}}
        \put(107,-15){$1$}
        \put(110,0){\circle*{2}}
        \put(-75, -15){$- \cos\frac{\pi}{q}$}
        \put(-55,0){\circle*{2}}

        \put(-75,83){$e^{i(\pi - \frac{\pi}{q})}$}
        \put(-55,95.5){\circle*{2}}

        \multiput(110,0)(0,5){41}{\line(0,1){2}}   

        \put(0,0){\arc(110,0){120}}
        \put(220,0){\arc(55,96){120}}

        \put(-55, 96){\line(0,1){104}}

        \put(275,96){\line(0,1){104}}
    \end{picture}
    \end{center}

    \caption{A fundamental region for $\Gamma_{q, \infty, \infty}$ \label{fig:fund_region}}
    \end{figure}

    By looking at a fundamental region of $\Gamma$ (figure \ref{fig:fund_region}) we see that this group has:
    \begin{enumerate}[(i)]
        \item 2 $\Gamma$-inequivalent cusps: $1$ and $\infty$
        \item 1 $\Gamma$-inequivalent elliptic point: $z_0 = e^{i(\pi - \frac{\pi}{q})}$.
    \end{enumerate}

    Moreover, it follows from the proof of Theorem 10.6.4 in \cite{beardon1983} that $\overline{\Gamma}_1 = \langle \overline{\gamma_2} \rangle$, $\overline{\Gamma}_{\infty} = \langle \overline{\gamma_3} \rangle$ and $\overline{\Gamma}_{z_0} = \langle \overline{\gamma_1} \rangle$. In particular, $|\overline{\Gamma_{z_0}}| = q$.

    \vspace\baselineskip

    Consider $\{ w_1, \dotsc, w_{k^{(1)}} \} = \varphi^{-1}(1)$ and let $e_1^{(1)}, \dotsc, e_{k^{(1)}}^{(1)}$ be their respective ramifications indices. Since $\Gamma(\fp) \unlhd \Gamma$, Proposition 1.37 in \cite{shimura1994iat} says that $e_1^{(1)} = \dotsb = e_k^{(1)} = \left[ \overline{\Gamma_1} : \overline{\Gamma(\fp)_1} \right]$ and $k^{(1)} e_1^{(1)} = \overline{\mu}$.

    $\overline{\Gamma}_1 = \langle \overline{\gamma_2} \rangle$ and $\overline{\Gamma(\fp)}_1 = \overline{\Gamma}_1 \cap \overline{\Gamma(\fp)}$. Since $\gamma_2^n = \left( \begin{array}{cc} - n + 1 & n \\ -n & n+1 \end{array} \right)$ and $\fp \cap \ZZ = p \ZZ$, then $\overline{\Gamma(\fp)}_1 = \langle \gamma_2^p \rangle$. So,
    \begin{equation}
        \label{eqn:1}
        e_1^{(1)} = \dotsb = e_k^{(1)} = p \textrm{ \ \ \ and \ \ \ } k^{(1)} = \frac{\overline{\mu}}{p}
    \end{equation}

    \vspace\baselineskip

    Let us now compute the ramification indices of $\varphi^{-1}(\infty)$. For this we need a claim (recall that $\mu_q = -2 - \zeta_{2q} - \zeta_{2q}^{-1}$):

    \vspace\baselineskip

    \begin{claim*}
        $\text{N}_{\QQ(\zeta_{2q}) / \QQ}(\mu_q) = q^2$.
    \end{claim*}

    In fact, notice that $\mu_q = - (1 + \zeta_{2q}) (1 + \zeta_{2q}^{-1})$. Since $q$ is odd, $-\zeta_{2q}$ is a primitive $q$-th root of unity. So, the minimal polynomial of $\zeta_{2q}$ is $\phi_{q}(-x)$, where $\phi_{q}$ is the $q$-th cyclotomic polynomial. So, the minimal polynomial of $1 + \zeta_{2q}$ is $h(x) = \phi_{q}(-(x-1)) = \phi_{q}(-x+1) = (-x+1)^{q-1} + \dotsb + (-x+1) + 1$. Thus, the constant term of $h$ is $q$. Hence, $\text{N}_{\QQ(\zeta_{2q}) / \QQ}(1 + \zeta_{2q}) = q$. Similarly, $\text{N}_{\QQ(\zeta_{2q}) / \QQ}(1 + \zeta_{2q}^{-1}) = q$. So, $\text{N}_{\QQ(\zeta_{2q}) / \QQ}(\mu_q) = (-1)^{q-1} \cdot \text{N}_{\QQ(\zeta_{2q}) / \QQ}(1 + \zeta_{2q}) \cdot \text{N}_{\QQ(\zeta_{2q}) / \QQ}(1 + \zeta_{2q}^{-1}) = q^2$.

    \vspace\baselineskip

    Hence, there exists $f(x) \in \ZZ[x]$ such that $\mu_q f(\mu_q) = q^2$ ($f(\mu_q)$ is the product of all the Galois conjugates of $\mu_q$ except for $\mu_q$ itself). Since $p \neq q$, this implies $\mu_q \not\in \fp$ (in fact, if $\mu_q \in \fp$, then $q^2 = \mu_q f(\mu_q) \in \fp$, which is impossible because $\fp \cap \ZZ = p \ZZ$).

    Now, since $\gamma_3^n = \left( \begin{array}{cc} 1 & n \mu_q \\ 0 & 1 \end{array} \right)$ and $\overline{\Gamma}_{\infty} = \langle \overline{\gamma_3} \rangle$, we see that $\overline{\Gamma(\fp)}_{\infty} = \langle \overline{\gamma_3}^p \rangle$. And, hence, $\left[ \overline{\Gamma}_{\infty} : \overline{\Gamma(\fp)}_{\infty} \right] = p$.

    Therefore, if $\{ v_1, \dotsc, v_{k^{(\infty)}} \} = \varphi^{-1}(\infty)$ and $e_1^{(\infty)}, \dotsc, e_{k^{(\infty)}}^{(\infty)}$ are their respective ramification indices, by Proposition 1.37 in \cite{shimura1994iat}, we get
    \begin{equation}
        \label{eqn:infty}
        e_1^{(\infty)} = \dotsb = e_{k^{(\infty)}}^{(\infty)} = p \textrm{ \ \ \ and \ \ \ } k^{(\infty)} = \frac{\overline{\mu}}{p}
    \end{equation}

    \vspace\baselineskip

    Now we shall compute the ramification indices of $\varphi^{-1}(z_0)$. We need to compute $\overline{\Gamma(\fp)}_{z_0}$. Since $\overline{\Gamma(\fp)}_{z_0} = \overline{\Gamma}_{z_0} \cap \overline{\Gamma(\fp)}$ and $\overline{\Gamma}_{z_0}$ has only elliptic elements (in addition to the identity), the next claim tells us that $|\overline{\Gamma(\fp)}_{z_0}| = 1$. Therefore, $\left[ \overline{\Gamma}_{z_0} : \overline{\Gamma(\fp)}_{z_0} \right] = q$.

    \vspace\baselineskip

    \begin{claim*}
        $\Gamma(\fp)$ has no elliptic element.
    \end{claim*}

    Since $z_0$ is the only inequivalent elliptic point and $\Gamma_{z_0} = \langle \overline{\gamma_1} \rangle$, we see that any elliptic element of $\Gamma$ is conjugate to some (non-trivial) power of $\gamma_1$. Since $\Gamma(\fp) \unlhd \Gamma$, if $\Gamma(\fp)$ contains an elliptic element, it would also contain some (non-trivial) power of $\gamma_1$. But since $\ord(\overline{\gamma}) = q$ is a prime, $\Gamma(\fp)$ would contain $\gamma_1$. But $\gamma_1 = \left( \begin{array}{cc} * & * \\ * & 0 \end{array} \right)$ and, hence, $\gamma_1 \notin \Gamma(\fp)$ ($1 \not\equiv 0 \pmod{\fp}$).

    \vspace\baselineskip

    Hence, if $\varphi^{-1}(z_0) = \{ y_1, \dotsc, y_{k^{(z_0)}} \}$ and $e_1^{(z_0)}, \dotsc, e_{k^{(z_0)}}^{(z_0)}$ are their respective indices, Proposition 1.37 in \cite{shimura1994iat} tells us that
    \begin{equation}
        \label{eqn:z0}
        e_0^{(z_0)} = \dotsb = e_{k^{(z_0)}}^{(z_0)} = q \textrm{ \ \ \ and \ \ \ } k^{(z_0)} = \frac{\overline{\mu}}{q}
    \end{equation}

    Using the Riemann-Hurwitz formula with the information given by (\ref{eqn:1}), (\ref{eqn:infty}) and (\ref{eqn:z0}) we get:
    \[
    \begin{array}{rcl}
        2g - 2 & = & -2 \overline{\mu} + \frac{\overline{\mu}}{p}(p-1) + \frac{\overline{\mu}}{p}(p-1) + \frac{\overline{\mu}}{q}(q-1) \\
                & = & \overline{\mu} \left( -2 + 2 - \frac{2}{p} + 1 - \frac{1}{q} \right) \\
                & = & \overline{\mu} \left(1 - \frac{2}{p} - \frac{1}{q} \right) .
    \end{array}
    \]
    Hence,
    \[
    g = 1 + \frac{\overline{\mu}}{2} \left( 1 - \frac{2}{p} - \frac{1}{q} \right) .
    \]
\end{proof}

\section{Genus of $X_{q, \infty, \infty}^{(0)}(\fp)$}

In this section, the genus of $X_{q, \infty, \infty}^{(0)}(\fp)$, where $\fp$ is a prime above $p$, is computed. It is assumed that $\overline{\Gamma_{q, \infty, \infty}} \big/ \overline{\Gamma_{q, \infty, \infty}(\fp)} \cong \PSL_2(\FF_{\fp})$ (which is always true when $q \geq 5$ according to Theorem \ref{thm:quotient_gp}) and $p \neq q$ are prime numbers strictly greater than $2$.

Using the Riemann-Hurwitz formula and the natural map
\[
	\varphi: X_{q, \infty, \infty}^{(0)}(\fp) \rightarrow X_{q, \infty, \infty}(1) = X_{q, \infty, \infty},
\]
it suffices to compute the ramification indices of $\varphi$.

Recall that
\[
    \left( \overline{\Gamma_{q, \infty, \infty}} \right)_{\infty} = \left\langle \gamma_3 = \begin{pmatrix} 1 & \mu_q \\ 0 & 1 \end{pmatrix} \right\rangle,
\]
\[
    \left( \overline{\Gamma_{q, \infty, \infty}} \right)_1 = \left\langle \gamma_2 = \begin{pmatrix} 0 & 1 \\ -1 & 2 \end{pmatrix} \right\rangle,
\]
\[
    \left( \overline{\Gamma_{q, \infty, \infty}} \right)_{z_0} = \left\langle \gamma_1 = \begin{pmatrix} -\lambda_q & -1 \\ 1 & 0 \end{pmatrix} \right\rangle.
\]

It can be shown, with the help of Proposition 1.37 in \cite{shimura1994iat}, that the monodromy of this map over
$\infty$ is given by the action of $\gamma_3$ on the set of cosets
$\overline{\Gamma_{q, \infty, \infty}} \left/ \overline{\Gamma_{q,
\infty, \infty}^{(0)}(\fp)} \right.$. Similarly, the monodromy of
this map over $1$ (resp. $z_0$) is given by the action of
$\gamma_2$ (resp. $\gamma_1$) on $\overline{\Gamma_{q, \infty,
\infty}} \left/ \overline{\Gamma_{q, \infty, \infty}^{(0)}(\fp)}
\right.$.

\begin{lemma}
    Let $\gamma \in \Gamma_{q, \infty, \infty}$. The action of $\gamma$ on $\overline{\Gamma_{q, \infty, \infty}} \left/ \overline{\Gamma_{q, \infty, \infty}^{(0)}(\fp)} \right.$ is equivalent to the action of $(\gamma \mod{\fp}) \in \PSL_2(\FF_{\fp})$ on $\PP^1(\FF_{\fp})$ via fractional linear transformations, i.e., the cycle decomposition of $\gamma$ (viewed as an element of the group of permutations of $\overline{\Gamma_{q, \infty, \infty}} \left/ \overline{\Gamma_{q, \infty, \infty}^{(0)}(\fp)} \right.$ is the same as the cycle structure of $(\gamma \mod{\fp})$ (viewed as an element of the group of permutations of $\PP^1(\FF_{\fp})$).
\end{lemma}
\begin{proof}
    The action of $\overline{\Gamma_{q, \infty, \infty}}$ on $\PP^1(\FF_{\fp})$ via linear fractional transformations is transitive (since $\overline{\Gamma_{q, \infty, \infty}} / \overline{\Gamma_{q, \infty, \infty}(\fp)} \cong \PSL_2(\FF_{\fp})$). Moreover, the stabilizer of $\infty \in \PP^1(\FF_{\fp})$ is $\Gamma_{q, \infty, \infty}^{(0)}(\fp)$. Hence, by group theory, the action of $\overline{\Gamma_{q, \infty, \infty}}$ on $\PP^1(\FF_{\fp})$ is equivalent to the action of $\overline{\Gamma_{q, \infty, \infty}}$ on $\overline{\Gamma_{q, \infty, \infty}} \left/ \overline{\Gamma_{q, \infty, \infty}^{(0)}}(\fp) \right.$.
\end{proof}

\begin{lemma}
\label{lem:monodromy_over_oo}
    The monodromy over $\infty$ is given by
    \[
        (0)(1 \dotsb p)(p+1, \dotsb, 2p) \dotsb (p^{f-1} + 1, \dotsb, p^f),
    \]
    where $\FF_{\fp} = \FF_{p^f}$. So, $\varphi^{-1}(\infty) = \{ w_0, w_1, \dotsc, w_{p^{f-1}}\}$ and
    \[
        \begin{array}{ccc}
            e_{w_0} = 1 & \textrm{ and } & e_{w_i} = p \textrm{, for } 1 \leq i \leq p^{f-1}.
        \end{array}
    \]
\end{lemma}
\begin{proof}
    Notice that $(\gamma_3 \mod{\fp}) = \begin{pmatrix} 1 & \beta \\ 0 & 1 \end{pmatrix}$, where $\beta = (\mu_q \mod{\fp}) \in \FF_{\fp} \backslash \{ 0 \}$ (the fact that $\beta \neq 0$ is part of the proof of Proposition \ref{prop:genus}).

    Hence, $\infty \in \PP^1(\FF_{\fp})$ is fixed by $(\gamma_3 \mod{\fp})$. Furthermore, since
    \[
        (\gamma_3 \mod{\fp})^n = \begin{pmatrix} 1 & n \beta \\ 0 & 1 \end{pmatrix}
    \]
    and $\car(\FF_{\fp}) = p$, all other points of $\PP^1(\FF_{\fp})$ generate an orbit of size $p$.
\end{proof}

\begin{lemma}
\label{lem:monodromy_over_1}
    The monodromy over $1$ has the same cycle decomposition. So, $\varphi^{-1} = \{ w_0, w_1, \dotsc, w_{p^{f-1}}\}$ and
    \[
        \begin{array}{ccc}
            e_{w_0} = 1 & \textrm{ and } & e_{w_i} = p \textrm{, for } 1 \leq i \leq p^{f-1},
        \end{array}
    \]
    where $\FF_{\fp} = \FF_{p^f}$.
\end{lemma}
\begin{proof}
    Notice that $(\gamma_2 \mod{\fp}) = \begin{pmatrix} 0 & 1 \\ -1 & 2 \end{pmatrix}$.

    It is easily seen that the only point of $\PP^1(\FF_{\fp})$ that is fixed by $(\gamma_2 \mod{\fp})$ is the point $1$.

    Now, consider the natural map
    \[
        \psi : X_{q, \infty, \infty}(\fp) \rightarrow X_{q, \infty, \infty}(1).
    \]
    Since $e_{v,g} = p$ for all $w = \psi^{-1}(1)$ (this is part of the proof of Proposition \ref{prop:genus}) and $\psi$ factors as
    \[
        \xymatrix{
            X_{q, \infty, \infty}(\fp) \ar[r] & X_{q, \infty, \infty}^{(0)}(\fp) \ar[r]^{\varphi} & X_{q, \infty, \infty}(1)
        },
    \]
    we have that $e_{w,f} = 1$ or $p$ for all $w \in \varphi^{-1}(1)$.

    The previous calculation says that there is only one point above $1$ having ramification degree $1$. Hence, the result follows.
\end{proof}

\begin{lemma}
    Let $\FF_{\fp} = \FF_{p^f}$ as before. The ramification behavior above $z_0$ is given as follows:
    \[
    	\varphi^{-1}(z_0) = \{ w_1, \dotsc, w_n, w_1', \dotsc, w_m'\},
    \]
    where
    \[
        \begin{array}{cccccc}
            e_{w_i} = q & , & e_{w_i'} = 1 & , & p^f + 1 = q n + m & , 
        \end{array}
    \]
    \[
        m = \left\{
                \begin{array}{ll}
                    0, & \textrm{if } p^f \equiv -1 \pmod{q} \\
                    2, & \textrm{if } p^f \equiv 1 \pmod{q},
                \end{array}
            \right.
    \]
\end{lemma}
\begin{proof}
    Notice that $(\gamma_1 \mod{\fp}) = \begin{pmatrix} \beta & -1 \\ 1 & 0 \end{pmatrix}$ for some $\beta \in \FF_{\fp}$.

    As in the proof of the previous lemma, $e_w = 1$ or $q$ for any $w \in \varphi^{-1}(z_0)$.

    Let $n$ denote the number of points whose ramification degree is $q$ and let $m$ denote those whose ramification degree is $1$.

    Then $m$ is also the number of points in $\PP^1(\FF_{\fp})$ fixed by $(\gamma_1 \mod{\fp})$. Hence $m \leq 2$.

    Since $\deg(\varphi) = p^f + 1$,
    \[
        p^f + 1 = n q + m.
    \]

    Since $q$ and $p$ are distinct primes, it follows that $m \neq 1$. Taking the previous equality $\mod{q}$, the precise value of $m$ (in terms of $(p^f \mod{q})$) follows.
\end{proof}

\begin{prop}
\label{prop:genus_of_X^0}
    The genus of the curve $X_{q, \infty, \infty}^{(0)}(\fp)$ is given by
    \[
        g = \frac{(q-1)}{2} n - p^{f-1},
    \]
    where $n$ and $f$ are as in the previous lemma.
\end{prop}
\begin{proof}
    Follows from the Riemann-Hurwitz formula applied to $\varphi$, the previous three lemmas and the fact that $g(X_{q, \infty, \infty}(1)) = 0$.
\end{proof}

\bibliography{takei}
\bibliographystyle{alpha}

\end{document}